\documentclass{amsart}
              [1997/02/02 v1.2e PROC Author Class]

\copyrightinfo{2006}
  {American Mathematical Society}

\newtheorem{theorem}{Theorem}[section]
\newtheorem{corollary}[theorem]{Corollary}

\begin{document}

\title{Separate continuity topology and a generalization of Sierpinski's theorem}

\author{V.V.Mykhaylyuk}
\address{Department of Mathematics\\
Chernivtsi National University\\ str. Kotsjubyn'skogo 2,
Chernivtsi, 58012 Ukraine}
\email{vmykhaylyuk@ukr.net}



\commby{Ronald A. Fintushel}



\begin{abstract}
The separately continuity topology is considered and some its properties are investigated. With help of these properties a generalization of Sierpinski theorem on determination of real separately continuous function by its values on an arbitrary dense set is obtained.
\end{abstract}

\maketitle
\section{Introduction}

It is well-known that every continuous function $f$ which defined on a topological space $X$ and valued in a Hausdorff space $Y$ totally determined by its values on a dense in $X$ set. The following question naturally arises: have the same property functions which are separately continuous? W.Sierpinski in [1] solved this problem for separately continuous functions $f:{\bf R}^2\to {\bf R}$. Further development of this result is closely connected with the investigation of joint continuity points set of separately continuous mappings and with the establishment of properties of the separately continuous topology. So, Z.~Piotrowski and E.~Wingler established in [2] that if every separately continuous function which defined on the product $X=X_1\times ...\times X_n$ and valued in the regular space $Z$ is almost continuous then every separately continuous function $f:X\to Z$ totally determined by its values on a dense in $X$ set. Recall that a function is called {\it almost continuous} if every nonempty preimage of open set has nonempty interior. Since every quasicontinuous function is almost continuous, it follows from result of paper [3] that if $X$ is a Baire space, $Y$ satisfies the first countable axiom and $Z$ is a completely regular space, then every separately continuous function $f:X\times Y\to Z$ totally determined by its values on a dense in $X\times Y$ set.

In this paper we study some properties of separately continuous topology and show that the condition of the completely regularity of $Z$ in the previous result can be weakened.

\section{Some notions}

Let $X$ and $Y$ be a topological spaces. We denote by ${\mathcal S}$ the collection of all subsets $S$ of the product $X\times Y$ which have the following property: {\it if $(x_0,y_0)\in S$ then there exist neighborhoods $U$ of $x_0$ and $V$ of $y_0$ in $X$ and $Y$ respectively such that $(\{x_0\}\times V)\bigcup (U\times\{y_0\})\subseteq S$}. Clearly that the system ${\mathcal S}$ forms a topology on the set $X\times Y$. Moreover, for every topological space $Z$ a function $f:X\times Y\to Z$ is continuous in this topology if and only if $f$ is separately continuous. This topology is  called {\it the separately continuous topology} or $s$-{\it topology}. For a set $A\subseteq X\times Y$ the its closure in the $s$-topology we denote by $[A]_s$.

Recall that the minimal cardinality of a family $\mathcal U$ of open nonempty sets $U\subseteq X$ such that for every neighborhood $V$ of $x$ there exists $U\in \mathcal U$ with $U\subseteq V$, is called {$\pi$-character $\pi_{\chi}(x,X)$ of  $X$ at $x\in X$}. The cardinal $\pi_{\chi}(X)=\sup \limits_{x\in X}\pi_{\chi}(x,X)$ is called {\it $\pi$-character of $X$}. Clearly that a space with the first countable axiom has at most countable $\pi$-character.

\section{Separately continuous topology, product topology and Sierpinski theorem}

We start from a result which show that the separately continuous topology is closely connected with the product topology.

\begin{theorem}\label{th:2.1} Let $X$ be a Baire space,  $Y$ be a topological space with the at most countable $\pi$-character, $H$ be an $s$-open set in $X\times Y$. Then there exists an open in $X\times Y$ set $G$ such that $[G]_s=[H]_s$.
\end{theorem}

\begin{proof} We put $F=[H]_s$, $G={\rm int} F$. Clearly that $G\subseteq F$, therefore $[G]_s\subseteq [F]_s=F$.

Now we show that $F\subseteq [G]_s$. Let $(x_0,y_0)\in H$ and $W$ be an open $s$-neighborhood of $(x_0,y_0)$. Since the set $H$ is $s$-open, there exists an open neighborhood $U$ of $x_0$ in $X$ such that $U\times \{y_0\}\subseteq H\bigcap W$. The space $Y$ has an at most countable $\pi$-character, therefore there exists an at most countable family $(V_n:n\in {\bf N})$ of open in $X$ nonempty sets which provides the countability of the $\pi$-character of $Y$ at $y_0$. We put  $A_n=\{x \in U: \{ x \} \times V_n \subseteq H\bigcap W \}$. The set $H\bigcap W$ is $s$-відкритою, therefore for every $x\in U$ there exists a neighborhood $V$
of $y_0$ in $Y$ such that $\{x\}\times V\subseteq H\bigcap W$. Thus there exists $n\in {\bf N}$ such that $x\in A_n$. That is $\bigcup\limits_{n=1}^{\infty}A_n=U$. But $X$ is a Baire space, therefore there exist an $n_0 \in {\bf N}$ and an open in $X$ nonempty set $U_1\subseteq U$ such that $U_1\subseteq \overline{A}$, where $A=A_{n_0}$. We put $V_1=V_{n_0}$. Take a point $(x_1,y_1)\in U_1\times V_{1}$ and an its $s$-neighborhood $W_1$. There exists a neighborhood $U_2$ of $x_2$ such that $U_2\times \{ y_1\}\subseteq W_1$. Note that $U_2\bigcap A \not= \O$, therefore $(U_2\times \{ y_1\} )\bigcap (A\times V_{1}) \not= \O$. Thus, $W_1\bigcap (A\times V_{1})\not= \O$. Hence, $(x_1,y_1)\in [A\times V_{1}]_s$, i.e. $U_1\times V_{1}\subseteq [A\times
V_{1}]_s$. But $A\times V_{1}\subseteq H$, and thus, $U_1\times V_{1}\subseteq [H]_s=F$. This means that $U_1\times V_{1}\subseteq {\rm int} (F)=G$. In particular, $A_1\times V_1\subseteq G$, where $A_1=A\bigcap U_1$. Moreover, $A_1\times V_1\subseteq A\times V_1 \subseteq H\bigcap W \subseteq W$. Therefore $G\bigcap W\not= \O$, i.e. $(x_0,y_0)\in [G]_s$. Hence, $H\subseteq [G]_s$, $F=[H]_s\subseteq [G]_s$ and the theorem is proved.
\end{proof}

Using this result we obtain the following generalization of Sierpinski's theorem on the case of separately continuous mappings of two variables.

\begin{theorem}\label{th:2.2} Let $X$ be a Baire space, $Y$ be a topological space with an atmost countable $\pi$-character, $Z$ be a topological space in which every two distinct points can be separated by closed neighborhoods. Then every separately continuous function $f:X\times Y\to Z$ totally determined by its values on a dense in $X\times Y$ set.
\end{theorem}

\begin{proof} Let $f:X\times Y\to Z$ and $g:X\times Y\to Z$ be separately continuous functions, moreover $f_{|_A}=g_{|_A}$, where $A$ is a dense in $X\times Y$ set. We prove that $f=g$.

Suppose the contrary. Thus there exists a point $p\in X\times Y$ such that $z_1=f(p)\not= g(p)=z_2$. Take open neighborhoods $W_1$ and $W_2$ of points $z_1$ and $z_2$ in $Z$  respectively such that $\overline{W_1}\bigcap \overline{W_2}=\O$. The sets $H_1=f^{-1}(W_1)$ and $H_2=g^{-1}(W_2)$ are $s$-open neighborhoods of $p$. Therefore the set $H=H_1\bigcap H_2$ is nonempty and $s$-open. Since $F=[H]_s\subseteq [H_1]_s\bigcap [H_2]_s$, $f(F)\subseteq f([H_1]_s)\subseteq \overline{W_1}$ and $g(F)\subseteq g([H_2]_s)\subseteq \overline{W_2}$. According to Theorem \ref{th:2.1}, there exists an open in $X\times Y$ set $G$ such that $[G]_s=F$. We have $f(G)\subseteq \overline{W_1}$ and $g(G)\subseteq \overline{W_2}$. Therefore $f(G)\bigcap g(G)=\O$. But since $f_{|_A}=g_{|_A}$, $G\bigcap A=\O$, a contradiction.
\end{proof}

\section{Nonregularity of separately continuous topology}

In the proof of Theorem \ref{th:2.2} we use an weaker result than Theorem \ref{th:2.1}. Namely, it is enough to have that ${\rm int}(F)\ne \O$. But, taking account that this conclusion can be true for every $s$-open set $H$, we obtain that the difference between these two conclusions is formal only. In this connection the following question naturally arises: is it true ${\rm int}(H)\ne \O$? Or, in other words, is every dense in $X\times Y$ set a $s$-dense set? The positive answer to this question give us the possibility to weaken the condition on the space $Z$ in Theorem \ref{th:2.2} to the Hausdorff condition. But the following result shows that the question formulated above has the negative answer.

\begin{theorem} \label{th:3.1} Let $X$ and $Y$ be infinite $T_1$-spaces such that there exist a pseudobase in $X\times Y$ with the cardinality no more than the cardinality of every open in $X$ or $Y$ sets. Then there exists a dense in $X\times Y$ set $C$ which is closed and nowhere dense in the separately continuous topology.\end{theorem}

\begin{proof} Without loos of generality we can suppose that a needed pseudobase $\mathcal P$ consists of open rectangle. Let $|{\mathcal P}|= \aleph$, $\alpha$ is the first ordinal of the cardinality $\aleph$, ${\mathcal P}=(P_{\xi}: \xi <\alpha)$, $P_\xi =U_\xi \times V_\xi$ for every $\xi <\alpha$, moreover $U_\xi$ and $V_\xi$ are open sets in $X$ and $Y$ respectively. We construct a sequence of points $(c_\xi=(a_\xi,b_\xi):\xi < \alpha )$ such that

(1) $c_\xi \in P_\xi$ for every $\xi <\alpha$;

(2) $a_\beta \ne a_\gamma$ and $b_\beta \ne b_\gamma$ for every $\gamma <\beta<\alpha$.

Let $\beta<\alpha$ and points $c_\xi$, $\xi<\beta$ are constructed. The sets $A=\{a_\xi :\xi<\beta \}$ and $B=\{b_\xi :\xi<\beta \}$ have the cardinality $\leq \aleph$. According to the conditions of the theorem we have $|U_\beta |\ge \aleph$ and $|V_\beta |\ge \aleph$. Therefore the sets $U_\beta \setminus A$ and $V_\beta \setminus B$ are nonempty. We choose arbitrary points $a_\beta \in U_\beta \setminus A$ and $b_\beta \in V_\beta \setminus B$. It easy to see that the sequence $(c_\xi: \xi\le \beta)$ satisfy the conditions (1) and (2). Thus, according to the transfinite induction there exists the required sequence.

We put $C=\{c_\xi :\xi<\alpha \}$. It follows from (1) that the set $C$ is dense in $X\times Y$. The condition (2) means that for every point $(x,y)\in X\times Y$ the sets $C^x=\{y'\in Y:(x,y')\in C\}$ and $C_y=\{x'\in X: (x',y)\in C\}$ are at most one-pointed in $T_1$-spaces. Therefore $C^x$ and $C_y$ are closed sets in $Y$ and $X$ respectively. Hence, the set $C$ is $s$-closed. Taking into account that the spaces $X$ and $Y$ have not isolated points, we obtain that the set   $C$ is not $s$-neighborhood of every its points. Thus, $C$ is not nowhere dense in the separately continuous topology.
\end{proof}

\begin{corollary} \label{c:3.2} The space $({\bf R}^2,s)$ is not regular.
\end{corollary}

\begin{proof} Let $C$ is a dense in ${\bf R}^2$ set, which is closed newhere dense set with respect to the separately continuous topology. We take an $s$-open set $H\subseteq {\bf R}^2\setminus C$. According to Theorem \ref{th:2.1} there exists an open in ${\bf R}^2$ set $G$ such that $[H]_s=[G]_s$. But $G\bigcap C\ne \O$, therefore $[H]_s\bigcap C\ne \O$. Thus, $({\bf R}^2,s)$ is not regular.
\end{proof}

Remark that condition on the space $Z$ in Theorem \ref{th:2.2} is weaker than the regularity. In particular, the space $({\bf R}^2,s)$ is not regular,
but every two points in it can be separated closed neighborhoods. Moreover, it follows from Theorem \ref{th:2.2} and \ref{th:3.1} that every continuous function on $({\bf R}^2,s)$ totally determined by its values on some closed nowhere dense set.

However, the following question naturally arises: is it possible to replace in the Theorem \ref{th:2.2} the condition on the space $Z$ to the regularity? The reviewer remarked that this question has the negative answer. Let $Z_1=Z_2=({\bf R}^2,s)$, $C$ is a closed nowhere dense in $({\bf R}^2,s)$ set which is dense in ${\bf R}^2$ with respect the usual topology (the existence of such set follows from Theorem \ref{th:3.1}), and $Z$ is a space which formed by conglutination of spaces $Z_1$ and $Z_2$ on the set $C$ with the factor-topology. Clearly that $Z$ is Hausdorff. But the identical mappings $f_1: ({\bf R}^2,s) \to Z_1$ and $f_2: ({\bf R}^2,s)\to Z_2$ are continuous mappings with values in $Z$, which coincides on $C$ and are different in other points.

\section{Acknowledgment}
The author would like to thank Maslyuchenko V.K. for his helpful comments.

\bibliographystyle{amsplain}

\end{document}